\documentclass[letterpaper, 10 pt, conference]{ieeeconf}
\IEEEoverridecommandlockouts
\usepackage[noadjust]{cite}
\usepackage{nicefrac}

\let\labelindent\relax
\usepackage{enumerate, enumitem, ulem}
\usepackage{changes}
\usepackage{xcolor}
\usepackage{multirow, lipsum}
\usepackage{enumerate}
\usepackage{booktabs} 
\providecommand{\keywords}[1]
{	
  \textbf{\textit{Keywords---}} #1
}
\usepackage{amsmath}
\usepackage{physics}
\usepackage{hyperref}
\usepackage{url}
\usepackage{graphicx}

\usepackage{subcaption}

\usepackage{siunitx}
\usepackage{import}
\usepackage{placeins}
\usepackage{amsthm}
\usepackage{amssymb}
\usepackage{empheq}
\usepackage{ulem}
\usepackage{algorithm}
\usepackage{float}
\usepackage{multicol}
\usepackage{cleveref}
\usepackage[english]{babel}
\usepackage{bbm}
\newtheorem{theorem}{Theorem} 
\theoremstyle{definition}

\newcommand{\RemoveForCamera}[1]{}

\newcommand{\indc}[1]{{\mathds{1}_{#1}}}

\newcommand{\RemoveCycle}[1]{}

\newtheorem{corollary}[theorem]{Corollary}
\usepackage{algorithm}
\usepackage{algpseudocode}
\usepackage{amsfonts}
\usepackage{algorithm,algpseudocode}
\usepackage{dsfont}
\usepackage{bbm}
\usepackage{amsmath}
\usepackage{bm}
\newtheorem{thm}{{\bf Theorem}}

\newtheorem{conjecture}{Conjecture}

\usepackage{amsmath,etoolbox}
\AfterEndEnvironment{equation}{\restarthintedrel}
\AfterEndEnvironment{align}{\restarthintedrel}
\newcounter{hints}
\renewcommand{\thehints}{\roman{hints}}
\newcommand{\hintedrel}[2][]{%
  \stepcounter{hints}%
  \if\relax\detokenize{#1}\relax\else\csxdef{hint@#1}{\thehints}\fi
  \mathrel{\overset{\textrm{(\thehints)}}{\vphantom{\le}{#2}}}%
}
\newcommand{\restarthintedrel}{\setcounter{hints}{0}}

\hypersetup{
    colorlinks=true,
    linkcolor=blue,
    filecolor=magenta,      
    urlcolor=cyan,
}

 \newcommand{\FocalNEAndOthers}[1]{}    
\newcommand{\hide}[1]{}

\newcommand{\bpi}{{\bm \pi}}

\newcommand{\eop}{\hfill $\Box$}                                                                                                                                                                                                                                                                                                                                                                                                                  
\usepackage{bm}

\usepackage{bm}

\usepackage{algorithm}
\usepackage{algpseudocode}
\usepackage{amsfonts}
\usepackage{algorithm,algpseudocode}
\usepackage{dsfont}
\usepackage{bbm}
\usepackage{amsmath}
\usepackage{bm}

\usepackage{empheq}

\newcommand{\floor}[1]{\lfloor #1 \rfloor}

\usepackage{textcomp}
\def\BibTeX{{\rm B\kern-.05em{\sc i\kern-.025em b}\kern-.08em
    T\kern-.1667em\lower.7ex\hbox{E}\kern-.125emX}}
\providecommand{\keywords}[1]
{
  \small	
  \textbf{\textit{Keywords---}} #1
}
\title{ 
Queue or lounge: strategic design for strategic customers
 }

\author{Riya Sultana, Khushboo Agarwal and  Veeraruna Kavitha  %
\thanks{R. Sultana and V. Kavitha are with Industrial Engineering and Operations Research, IIT Bombay, Powai, Mumbai, 400076, India
        {\tt\small \{riya{\_}sultana, vkavitha\}@iitb.ac.in}}%
\thanks{K. Agarwal is with Inria Sophia Antipolis, 2004 Route des Lucioles, Valbonne 06902, France
        {\tt\small agarwal.khushboo@inria.fr}}%
}
    
\begin{document}
\maketitle
\thispagestyle{empty}
\pagestyle{empty}

\begin{abstract}
Considering an M/M/1 queue with an additional
lounge facility (LF), the quest of this paper is to understand
the instances when LF is an attractive option, from customer
perspective as well as from  system perspective – will the
customers choose to join the queue or prefer to detour briefly
to lounge? In reality, customers do not perform complex
computations for such tasks, but instead choose based on some
heuristics. We further assume that the customers pessimistically
anticipate the future congestion while making the choice. Our
analysis reveals that the customers use the LF only when
the queue is too crowded, and the lounge is relatively empty;
however, strikingly, the customer choice is more inclined towards
rejection for the LF in systems with higher traffic (load).

We also explore an optimization problem where the system
determines whether to implement an LF and what capacity
it should have, while accounting for customers’ behavioral
responses. Under low load conditions, the system benefits from
designing a high-capacity lounge, and the customers also prefer
to use the LF actively. Surprisingly, neither the system prefers
big LF, nor the customers prefer to use the LF
profusely at high load conditions; optimal for either is to use the
LF sparingly. Thus, importantly, the strategic system and the
bounded-rational customers are not in a tug-of-war situation.

\end{abstract}

\section{Introduction}
Customers often encounter the choice to join or not to join a queue in real-life. Naor in \cite{naor1969regulation} first analyzed this question for strategic customers. This study ignited an extensive research on \textit{strategic queueing} where customers can strategically decide to join the queue, balk (not to join the queue at all), renege (leave the queue while waiting in the queue), jockey (switch between queues), etc. For comprehensive coverage of this field, refer to \cite{hassin2003queue, hassin2016rational} and other works like \cite{honnappa2015strategic, altman2005applications, bendel2018cooperation}.  

A particular system which interests us is famously known as retrial queueing system. Here, customers facing long queues may temporarily leave for a random amount of time (in what is called an `orbit') before attempting to rejoin, thus avoiding extended waiting times.
Several variants of retrial queues have been studied before, see, e.g., \cite{kerner2020strategic, wang2013strategic, cui2019model, avrachenkov2008retrial}. 
A key feature of retrial queues in general is that the size (the number of customers waiting) of the orbit is non-observable to the customers. An exception to this point is the work in \cite{wang2013strategic} where customers can observe both server status (idle/busy) and orbit size; however, in this case, there is no queue as the customers either receive immediate service or wait in the orbit. Our work differs precisely at this point by introducing a system with both a queue and a lounge facility (LF), which are observable to the customers at their respective arrival instants. 

The concept of an LF is commonly seen in airports, visa centers, movie theaters, hospitals, etc. In fact, the authors in \cite{juneja2016lounge} conducted a  study of a single-queue, single-server system with an LF, developing a threshold-based policy for customers who begin in the lounge and then strategically decide whether to join the queue. Our work differs by giving customers the freedom to choose between the queue and the lounge before entering the system.

Specifically, we examine an M/M/1 queuing system with a lounge facility, where arriving customers have the flexibility to join the queue directly or wait in the lounge for an exponentially distributed time before joining the queue.  Once made, the decisions can not be changed, and the customers can not balk. The service is provided in the order of joining the queue.

Importantly, we depart from the standard strategic queuing literature by considering bounded rational customers, recognizing that perfect rationality is often unrealistic in real-world contexts (see \cite{camerer2011behavioral, sandholm2010population, singh2024stochastic, agarwal2025two}). Specifically we assume each arriving customer to choose between the two options (queue or lounge), after estimating the  costs corresponding to the two options using some simple calculations (e.g., as in \cite{mckelvey1995quantal}). 
Often people tend to estimate the rate of occurrences while drawing conclusions about completion time   of a future event;
 we assume the customers to compute the waiting delays in either option using fluid rates. We also consider that the customers  pessimistically anticipate all future arrivals and the existing lounge occupants to join the queue before them, while exploring the LF option. This pessimal anticipation approach has precedent in works like \cite{aumann1961core, sultana2024cooperate, singh2024stochastic}.

For these reasons, we believe our framework offers novel contributions in terms of both system architecture and customer responses. We prove that the customers utilize the LF only when the observed queue is long and the lounge congestion remains low.
While accounting for the behavioral responses of the customers, we analyze a bi-level optimization problem for the system which aims to optimally design the LF for managing the congestion in lounge and queue. We conclude:
\begin{enumerate}
    \item[(i)] in low load settings, the system prefers to design the LF with high capacity and customers possess high tolerance for joining the LF, 
    \item[(ii)] surprisingly, under high load conditions, the optimal system design features low-occupancy lounge, and customers also become less willing to use the LF. 
\end{enumerate}
Notably, contrary to the expectation that system and customers' interests would conflict, we find that their optimal responses actually align.

\section{Queueing system with lounge facility}
We consider a queueing system consisting of a single server and a single queue, enhanced with a lounge facility (LF).  This system models scenarios where the arriving customers have the flexibility to either immediately join the queue or wait in a  lounge area before entering the queue. 

The customers arrive in the system following a Poisson process with rate $\lambda > 0$; thus the inter-arrival times between consecutive customers are independently and exponentially distributed. The service times are also assumed to be exponentially distributed with the rate 
$\mu > 0$. Upon arrival, each customer independently decides whether to enter the queue immediately (and wait for its turn to be served by the server) or wait in the lounge before joining the queue. The lounge serves as a temporary waiting area where  some customers may prefer to relax for sometime depending upon the occupancy levels. Customers choosing LF do not receive service immediately, but   join the queue after spending a random amount of time in the lounge. 
The duration of a customer's stay in the lounge is assumed to be exponentially distributed with rate 
$\nu > 0$.  This rate $\nu$ is assumed to be fixed,   reflects the customer behavior pattern and is considered not controllable. 
Further, we assume the following for the obvious reason of enhanced   efficiency:
\begin{enumerate}[label=\textbf{A.\arabic*}, ref=\textbf{A.\arabic*}]
    \item If the queue is empty and the lounge is occupied, then a customer immediately moves from the lounge to the queue. \label{a1}
\end{enumerate} 
It is assumed that the movement from lounge to queue takes no time. It is natural that customers react to such a facility (LF) in a strategic manner, probably depending upon the available information; we defer the discussion on this aspect to the next section. 

Define $\rho := \nicefrac{\lambda}{\mu}$. It is well-known that one requires $\rho < 1$ for a standard M/M/1 queue, with no LF. Interestingly, the same condition suffices to ensure   stability even for the system with LF (details are in sub-section \ref{subsec_MM1}),   hence we  assume:
\begin{enumerate}[label=\textbf{A.\arabic*}, ref=\textbf{A.\arabic*}]
\setcounter{enumi}{1}
    \item Consider $\rho < 1$. \label{a2}
\end{enumerate}
Once a customer enters the queue, they receive service following the First-Come, First-Served (FCFS) discipline. 

While a customer waits either in the queue or the lounge, it incurs some cost. In particular, while waiting in a queue, a customer incurs $\alpha$ cost, and in the lounge, it incurs a smaller $\beta$ cost per unit time; we assume $\alpha > \beta > 0$. Importantly, even the customers joining the queue after waiting in the lounge incur the cost at rate $\alpha$ for the additional time spent in the system.


If the customers in the system were naive, they may have a bias towards a choice (queue or lounge), they may simply follow others, etc. However, in reality, it is not uncommon that customers do simple calculations and comparisons while making such decisions. They may choose the LF, more so when congestion is high, in the hope of returning to a shorter queue.
In particular, they may also 
attempt to balance the immediate convenience of the lounge against the potential risk of facing a longer queue upon their return. We  capture such a behaviour of the customers in the immediate next. 
\section{Decisions of bounded-rational customers}
For the system with LF, each customer faces a binary choice upon arrival. Let  this choice  be represented by $a \in \{1, 2\}$, where,
\begin{enumerate}
     \item[(i)] $a=1$ denotes the action to directly join the queue, and
    \item[(ii)] $a = 2$ denotes the action to first enter the lounge and then the queue.
\end{enumerate}
This decision is guided by the trade-off between the anticipated waiting costs associated with either choice. We assume that each customer knows the system parameters ($\alpha, \beta, \lambda, \mu$ and $\nu$). Further, the queue length (denoted by $Q$) and the lounge size (i.e., the number of customers waiting in the lounge, denoted by $L$) are fully observable to the arriving customer. The knowledge about the lounge size differentiates our work from the typical retrial queues where the size of the orbit is unobservable to the customers (\cite{avrachenkov2008retrial, kerner2020strategic, cui2019model}).



\subsection*{Bounded-rational behaviour and pessimal anticipation}
Any strategic customer aims to choose an option that minimizes their total expected waiting cost in the system. However, queueing systems involve a large number of customers and it is unreasonable to assume that everyone is completely rational (see \cite{agarwal2025two, mckelvey1995quantal, camerer2011behavioral, sandholm2010population}). Therefore, we incorporate bounded rationality into our modeling to capture more realistic patterns of customer behaviour.

In practice, customers rarely perform exact computations of expected waiting times, as such calculations require detailed consideration of the service mechanism, arrival process, and decisions of  other customers. Instead, customers rely on heuristic or approximate reasoning to guide their decisions. For instance, it is common to hear statements like ``in the last ten minutes, about eight people left the queue — things are moving faster now", or ``it has been five minutes and only one person moved — this might take a while", or even ``the queue looked shorter a few minutes ago, maybe more people are joining now". 

This kind of reasoning suggests that customers estimate queue evolution process via  fluid estimates constructed using  the most recent observations. 
For example, an incoming customer can anticipate the  queue to deplete at an effective rate of
$\mu$, when it considers joining the queue directly; hence if it observes  $Q$ number of waiting customers in the queue upon arrival,  its anticipated/expected  waiting time  (before its turn), by directly joining the queue ($a=1$), is simply $\nicefrac{Q}{\mu}$ and hence the anticipated cost is,
\begin{equation}
    C(a=1) = 
\alpha\frac{Q}{\mu}. \label{eq_estimatedcost1}
\end{equation}

On the other hand, if a customer considers relaxing in lounge, 
it might anticipate the queue to deplete at a different rate; for example,  if it anticipates all the future arrivals to join the queue directly then the effective (fluid) rate is $(\mu-\lambda)$. 
The latter kind of  anticipation reflects yet another commonly observed behavioral tendency, that of pessimal or worst-case anticipation --- all future arrivals join the queue directly increasing its waiting time (for instance, remarks like 
``we are approaching the closing time, and there are already  a large number of customers, hence better to return tomorrow'' reflect decisions based on worst-case anticipation). 

Further, the incoming customer considering the lounge decision, can also  (pessimistically) anticipate all the $L$ customers in the lounge to join back the queue before it. 
So, considering the fluid estimates and  both the pessimal anticipations (worst depletion rate $(\mu-\lambda)$ and all   $Q+L$ customers including those in the lounge getting served before the tagged customer), the 
anticipated/expected  waiting time (before the server gets empty), for an incoming customer considering the lounge decision,  equals $\nicefrac{(Q+L)}{(\mu-\lambda)}$. However,  the customer   does not plan to wait in the lounge till the end, instead plans to return after time $\nicefrac{1}{\nu}$;  then it anticipates $\nicefrac{(\mu-\lambda)}{\nu}$ number of customers to be served before its return to the queue;  in other words, it anticipates  the following queue size at the instance of joining back the queue:
\begin{eqnarray}
    Q+L - \frac{\mu-\lambda}{\nu}, \label{eq_first}
\end{eqnarray}

or zero if the above is negative.

In all, the anticipated cost of joining the queue later ($a=2$) is given by:
\begin{equation}
    C(a=2) = \frac{1} {\nu} \beta + \frac{\alpha}{\mu}  \left ( Q+L - \frac{\mu-\lambda}{\nu}  \right )^+. \label{eq_estimatedcost2}
\end{equation}

To summarize, we model the bounded-rational decisions of the customers by replacing the performance of the actual     queuing system  with  that of the fluid process and by considering pessimal anticipation for two aspects. 
At this point, we would like to draw the attention of the reader to the fact that the above usage of pessimal anticipation rule is not new. It has been previously utilized in several other game-theoretic contexts (e.g., \cite{sultana2024cooperate,  aumann1961core} on cooperative games, \cite{singh2024stochastic} while modeling vaccination decisions) to once again model the anticipated utilities of various players of the game.

Thus given the  state $(Q,L)$ 
at the arrival,  the concerned customer makes an independent decision
depending upon  two estimated waiting costs  \eqref{eq_estimatedcost1} and \eqref{eq_estimatedcost2}:  joins the lounge if and only if (with tie-breaker in favor of $a = 1$):
\begin{align}
\label{Eqn_Customer_decision} 
C(a=1) > C(a=2)
\end{align}


\subsection{Decision rule of customers}
The state dependent comparison of the two anticipated costs \eqref{eq_estimatedcost1} and \eqref{eq_estimatedcost2} describe  the dynamic decisions of the customers. 
We consider this analysis in the immediate next, which indicates  that the customers adopt a two-dimensional \textit{threshold-based policy}:

\hide{
These values represent:

\begin{itemize}
    \item \textbf{Threshold \( A \) (Queue Length Threshold):}  
    This threshold determines when the \textbf{queue becomes too long}, prompting customers to opt for the lounge instead of joining the queue immediately. If the number of customers in the queue exceeds \( A \), the estimated waiting cost in the queue becomes larger than the lounge waiting cost, making the lounge the preferred choice.
    
    \item \textbf{Threshold \( B \) (Lounge Occupancy Threshold):}  
    This threshold defines the maximum number of customers in the lounge before \textbf{customers start preferring the queue over the lounge}. If the number of customers in the lounge exceeds \( B \), the anticipated congestion upon returning from the lounge is too high, leading customers to directly enter the queue instead.{\color{red} the functions need monotonicity}
\end{itemize}
}

\begin{thm} \label{lemma_threld}{\bf [Response of customers]}
    The optimal decision rule $a^*$ for an arriving customer based on the anticipated costs \eqref{eq_estimatedcost1} and \eqref{eq_estimatedcost2} is given by:
\begin{align}\label{eqn_threshold}
a^* = 
\begin{cases}
    2, \mbox{ if } Q > A \mbox{ and } L < B, \\
    1, \mbox{ otherwise,}
\end{cases}
\end{align}where the constants:
\begin{equation}
    A = \frac{\mu \beta}{\alpha \nu} \mbox{ and }
    B = \frac{\mu -\lambda}{\nu}- A.
\end{equation} 

\end{thm}
\begin{proof}
%
First consider 
$Q +L\le \nicefrac{(\mu - \lambda)}{\nu}$, so  $C(a=2) = \nicefrac{\beta}{\nu}$.  
 The customer joins the lounge only if,
$  Q > \nicefrac{\mu \beta}{\alpha \nu}$, see   \eqref{Eqn_Customer_decision}; 
 also observe  $L < B$. Now, let
$Q +L > \nicefrac{(\mu - \lambda) }{\nu}$. Then,   the customer joins the lounge if,

\vspace{-2mm}
{\small$$
    \frac{\alpha Q}{\mu} > \frac{\beta}{\nu}+\frac{\alpha}{\mu} \left(Q +L- \frac{\mu - \lambda}{\nu}\right)   \implies  L < B; 
$$}
and now observe, $
    Q > \nicefrac{\mu \beta}{\alpha \nu}.
$
\end{proof}
Thus   a strategic customer chooses to enter the lounge only when the queue is sufficiently long ($Q>A$) and the lounge is not overly crowded ($L<B$). 
This is not a very surprising result, nonetheless,  interestingly,  this is obtained after careful consideration of bounded rational and pessimal attributes  of the customers. The threshold policy is depicted pictorially in \autoref{fig:threshold}.

\begin{figure}[htbp]
\includegraphics[trim = {0.4cm 0.9cm 0cm 0.6cm}, clip,scale=.45]{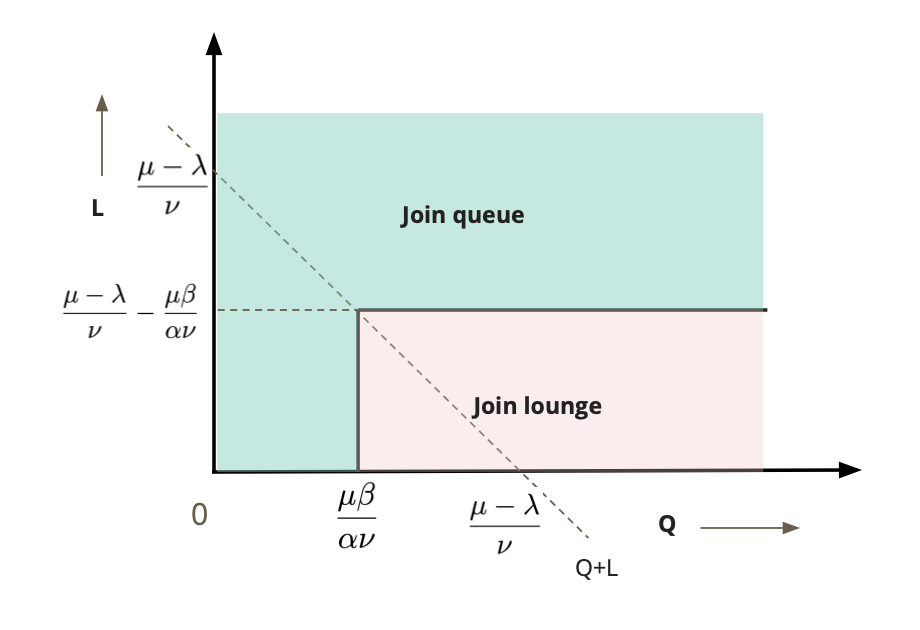}
     \caption{Decision-rule for bounded-rational customers}
    \label{fig:threshold}
    \end{figure}

\begin{figure*}
\hspace*{8pt}
    \begin{minipage}{0.6\textwidth}
        \flushleft
    \includegraphics[trim = {1.3cm 0cm 3cm 0cm}, clip,scale=.16]{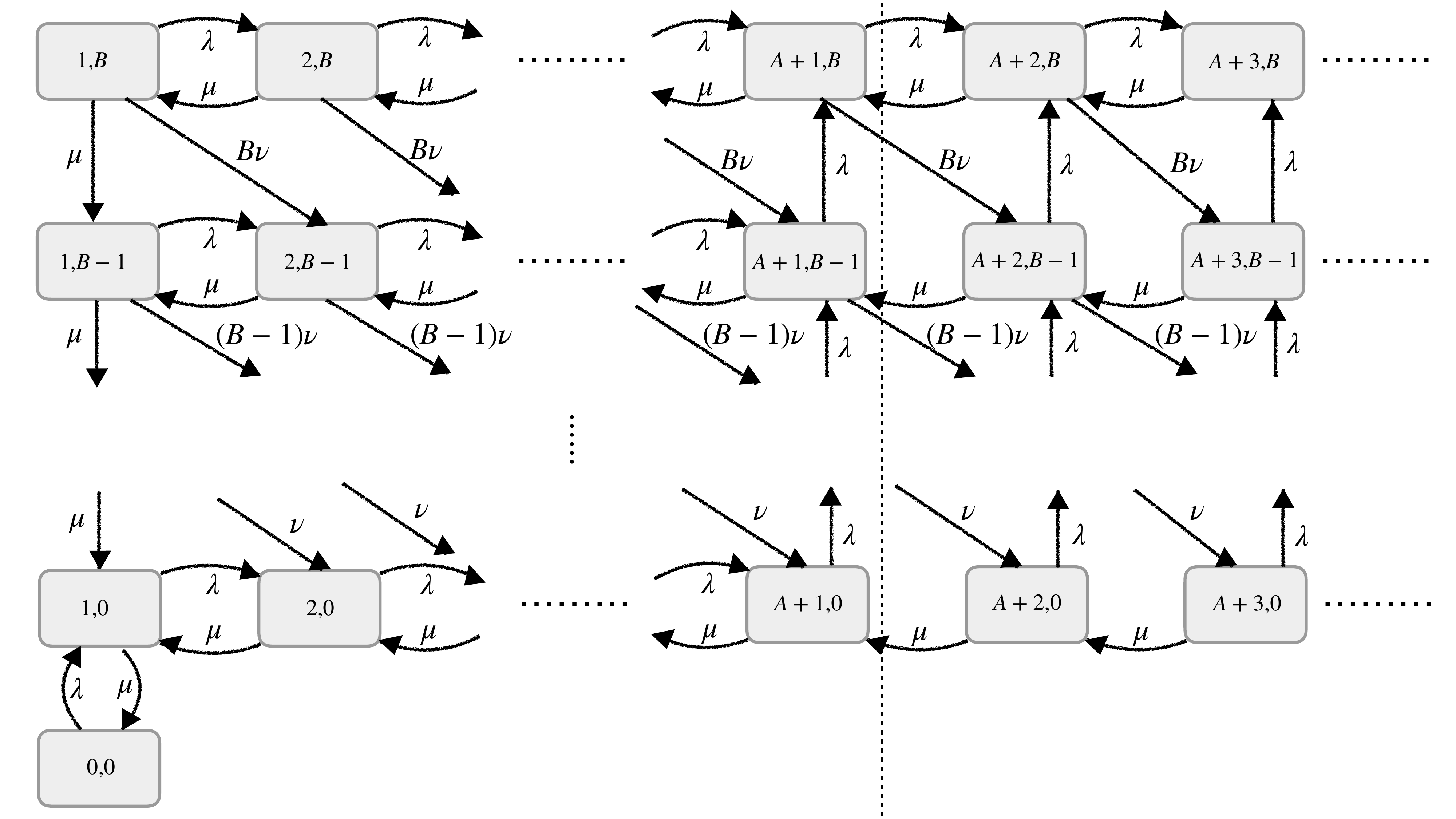}
    \caption{Transition rate diagram for the system with LF}
    \label{fig:transition}
    \end{minipage}%
    \hspace{-8mm}
    \vspace{-4mm}
    \begin{minipage}{0.4\textwidth}
    \vspace{1.4cm}
        \flushright
        \includegraphics[trim = {14cm 0cm 14cm 9cm}, clip, scale = 0.16]{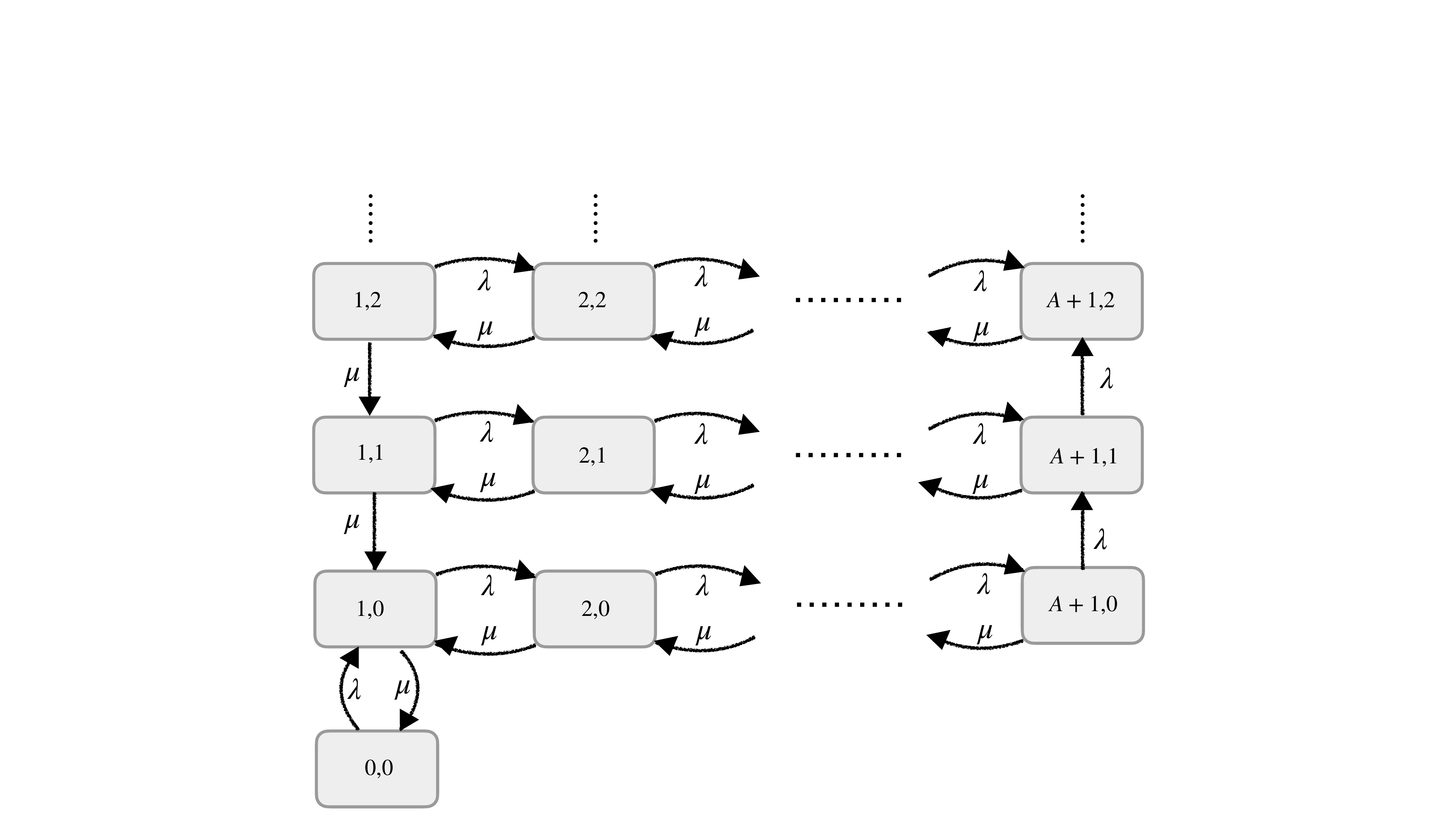}
        \caption{Transition rate diagram for the approximating system}
        \label{fig:transition2}
    \end{minipage}
\end{figure*}


\subsection{LF for higher load factors: customer perspective}
A key metric in queueing systems is the load factor $\rho$, which represents the number of customers arriving in one service period.
As already mentioned, 
we would require $\rho < 1$ for stability (see  section \ref{sec_analysis}). 
\textit{The viability of LF option is probably a more important question for systems} with large congestion costs, that is for  the systems \textit{with high load factors.} Thus we now investigate this aspect from customers' perspective, while  discussion on system perspective is postponed to section \ref{sec_loungedesign}. 
To be precise, we  investigate if an arriving customer would join the  lounge when $\rho$ is close to $1$, which is answered in the immediate next.


\begin{corollary}
    Let $\Delta := 1-\nicefrac{\beta}{\alpha}$. 
 The customers always choose to join the queue  if $\rho \ge \Delta.$
    
\end{corollary}
\begin{proof}
    Clearly, $B \leq 0$; hence, the result by   Theorem~\ref{lemma_threld}.
\end{proof}
Thus, for  high load factors ($\rho \geq \Delta$) all the customers join the queue, regardless of the queue-lounge occupancy levels at their arrival instance.  
From equation \eqref{eq_first},  not much reduction in the congestion levels is anticipated by the customer (while they relax in lounge), and hence possibly the reason for this 
counter-intuitive result.
However the customers can be lured to LF by increasing the comfort levels in the lounge, i.e., by reducing $\beta$ and hence  increasing $\Delta$. In all, at high load factors, customers may accept LF only if  it is perceived to be exceptionally comfortable.

Before proceeding further, we would like to make an important remark. Constants $A$ and $B$ of Theorem \ref{lemma_threld} are real numbers, however without loss of generality, one can consider them as the following integers, as queueing system evolve over integers:
\begin{eqnarray}
  A =\lfloor\nicefrac{\mu \beta}{\nu \alpha}\rfloor, \mbox{ and } B=\lceil \nicefrac{(\mu-\lambda)}{\nu}- A\rceil.
\end{eqnarray}

\section{Analysis}\label{sec_analysis}
This section studies the continuous-time Markov chain (CTMC)  representing the underlying queuing system. The state space  is given by $\mathbb{S} := \{(Q, L): Q, L \in \mathbb{Z}^+\cup\{0\}\}$. Under the threshold policy given in Theorem \ref{lemma_threld}, when $L = B$, only two possibilities can occur: (i) a new customer arrives and joins the queue, or (ii) a previously waiting customer in the lounge returns to the queue. Notice that the lounge size can not exceed $B$ in either scenario. Thus, at most $B$ customers can use the infinitely capacitated lounge due to the decisions of the strategic customers. Therefore, the state space reduces to the following:
$$
\mathbb{S} = \{(Q, L): Q\in \mathbb{Z}^+\cup\{0\}, L \in \{0, 1, \dots, B\}\}.
$$



Let $(x, y)$ be the state of the system. Then, the rate at which the system transitions into the new state $(x', y')$ is given by $r_{(x,y) \to (x', y')}$, where: i) $r_{(0,0) \to (1,0)} = \lambda$, $r_{(1,y) \to (1,y-1)} = \mu$, 
    $r_{(1,y) \to (0,y)} = 0$ (by \ref{a1}), for all $y > 0$; 
(ii) for all $x \geq 1$, 
\begin{align*}
    r_{(x,y) \to (x', y')} &= 
    \begin{cases} 
        \mu, &\text{if } x'=x-1, y'=y, y \geq 0\\
        y\nu, &\text{if } x'=x+1, y'=y-1, y > 0,
    \end{cases};
\end{align*}and (let $\indc{x > A,  y<B}=\xi$),
\begin{align*}
    r_{(x,y) \to (x', y')} &= 
    \begin{cases}
        \lambda \xi, &\text{if } x'=x, y'=y+1, \\
         \lambda(1-\xi), &\text{if } x'=x+1, y'=y, y \geq 0.
    \end{cases}
\end{align*}

\hide{(ii) for all $x \geq 1$, 
\begin{align*}
    r_{(x,y) \to (x', y')} &= 
    \begin{cases} 
        \lambda, &\text{if } x'=x+1, y'=y, y \geq 0\\
        \mu, &\text{if } x'=x-1, y'=y, y \geq 0\\
        y\nu, &\text{if } x'=x+1, y'=y-1, y > 0,
    \end{cases};
\end{align*}and (iii) for all $x> A, y< B$,
\begin{align*}
    r_{(x,y) \to (x', y')} &= 
    \begin{cases}
        \lambda, \text{ if } x'=x, y'=y+1, \\
        0, \text{ if } x' = x+1, y' = y.
    \end{cases}
    \end{align*}}
The second relation in the above holds by  Theorem \ref{lemma_threld} and all other rates can be easily verified from  the system description. The transition diagram of the CTMC is provided  in \autoref{fig:transition}.

Next, we derive the stationary distribution for the underlying process. Towards this, we first make an important observation. 

\subsection{M/M/1 connection and stability}\label{subsec_MM1}
Under assumption \ref{a1}, the queue is never empty if the lounge is occupied. Therefore, the total number of customers in the queue and lounge, i.e., $Q+L$, exactly equals the number of customers in an  M/M/1 queue. Consequently, under \ref{a1} and \ref{a2}, we have two important outcomes:
\begin{enumerate}
    \item[(i)] the system with LF is stable and possesses a unique   stationary distribution $\pi$;
    \item[(ii)] let $\pi_{i, j}$ represent the stationary probability of  $i$ and $j$ customers waiting in the queue and the lounge respectively; then the stationary probability of having exactly $n$ customers in the system (regardless of their individual locations) is given by:
    \begin{eqnarray}
    \hspace{-4mm}
       \pi(Q+L = n) =  \sum_{l \le B} \pi_{n-l, l}  =  \rho^n (1-\rho)  \mbox{ for all } n.\label{eqn_prob_n_cust}
    \end{eqnarray}
\end{enumerate}

\subsection{Stationary distribution for $B=1$} \label{staionary_dist}

From Theorem  \ref{lemma_threld},  for high load factors, which  is of particular importance to us, the value of $B=\lceil\nicefrac{(\mu-\lambda)}{\nu}-A\rceil$ is small and close to $0$; the stationary analysis of such a system  (equivalently with $B=1$) will be instrumental in further analysis (in section~\ref{sec_loungedesign}) and hence is considered here (ironically the analysis is simple enough only for $B=1$, see sub-section \ref{subsec_approx}
for approximate analysis with $B > 1$).


From \eqref{eqn_prob_n_cust},  the stationary probability $\pi_{q, 1}+\pi_{q+1,0}$ of the  total number of customers in the system equals that in M/M/1  
queue for any~$q$. It remains to derive the relation between the two terms.  As $q \to \infty$, one might expect the ratio between the two of them (i.e., $\pi_{q+1, 0}, \pi_{q, 1}$) to converge to a constant. Formalizing this intuition, we set $\pi_{q+1,0} = c \pi_{q,1}$ for all large enough $q$ (specifically for $q \geq A+1$) and attempt to solve the balance equations, which for this sub-case are given by (see Figure \ref{fig:transition}, with $\theta :=\lambda+\mu+\nu$):
\begin{eqnarray*}
      \pi_{q+2,0} &=& \frac{\theta}{\mu}\pi_{q+1,0}-\frac{\nu}{\mu}\rho^{q+1}(1-\rho) , \mbox{ and}\\
    \pi_{q+1,1} &=& \frac{\theta}{\mu}\pi_{q,1}-\frac{\lambda}{\mu}\rho^{q}(1-\rho), \mbox{ for } q \geq A+1. 
\end{eqnarray*}
By substituting $ \pi_{q+1,0} = c \pi_{q,1}$ relation into the above balance equations, we found that our intuition is indeed correct and $c$ is uniquely given by $\nicefrac{\nu}{\mu}$ due to uniqueness of the stationary distribution (argued in sub-section \ref{subsec_MM1}). To summarize, we have:
\begin{align}
    \frac{\pi_{q+1,0}}{\pi_{q,1}}=\frac{\nu}{\mu}, \mbox{ for } q \geq A+1. \label{eqn_ratio}
\end{align}
Using \eqref{eqn_ratio} and \eqref{eqn_prob_n_cust} with $n = q+1$, we have the following stationary probabilities for all $q \geq A+1$:
\begin{align}
    \pi_{q+1,0}= \frac{\rho^{q+1}(1-\rho)\nu}{\nu+\mu} \mbox{ and }
    \pi_{q,1}= \frac{\rho^{q+1}(1-\rho)\mu}{\nu+\mu}. \label{eq_statprob1}
\end{align}
We are now left to derive the stationary probabilities for $q < A+1$.  First, we focus on the case where $l = 1$. For this sub-case, the balance equations are given by (see Figure \ref{fig:transition} for $q\in \{2, \cdots, A\}$,
\begin{eqnarray}
\theta 
\pi_{q,1} =\mu \pi_{q+1,1}+ \lambda \pi_{q-1,1},  \mbox{ and, }  \theta\pi_{1,1} = \mu \pi_{2,1} . 
\end{eqnarray}
Such recursive equations can be solved using standard techniques, for example using  \cite[equation (1), pg. 4]{rrecursionbook}  we obtain the following:
\begin{align}
    \pi_{q,1}&=m_{q-1}\pi_{1,1}, \label{eq_statprob2}
\end{align}where the constant $m_q$ is defined  recursively as follows:
\begin{align*}
    m_q=\frac{\theta}{\mu}m_{q-1}-\rho m_{q-2} \mbox{ with }m_1=\frac{\theta}{\mu} \mbox{ and } m_0=1.
\end{align*}
The above recursion is solved and the solution is given by:
\begin{align}
    m_{q}&= \frac{\frac{\theta}{\mu}-\beta}{\alpha-\beta} \alpha^q+\frac{\alpha-\frac{\theta}{\mu}}{\alpha -\beta} \beta^q, \mbox{ for}\nonumber\\
\alpha&=\frac{\frac{\theta}{\mu}+\sqrt{\left(\frac{\theta}{\mu}\right)^2-4\rho}}{2} \mbox{ and } \beta=\frac{\frac{\theta}{\mu}-\sqrt{\left(\frac{\theta}{\mu}\right)^2-4\rho}}{2}. \nonumber
\end{align}
Now, using \eqref{eq_statprob1} and \eqref{eq_statprob2} for $q = A+1$, we obtain $\pi_{1,1}$ as:
\begin{align}\label{eqn_pi11}
    \pi_{1,1}= \frac{\rho^{A+2}(1-\rho)\mu}{(\nu+\mu)m_{A}};
\end{align}one can then substitute the above in \eqref{eq_statprob2} to obtain $\pi_{q, 1}$ for all $1 < q \leq A+1$.

For $l=0$  and $q \le A+1$ from \eqref{eqn_prob_n_cust}, we have
\begin{align}
    \pi_{q,0}&=\rho^q(1-\rho)-\pi_{q-1,1}.\label{eq_lasteq}
\end{align}
Using \eqref{eq_statprob1}-\eqref{eq_lasteq}, the 
stationary distribution for the $B=1$ can be summarized as follows (where $\psi := \nicefrac{(1-\rho)}{(\mu + \nu)}$):
\begin{align}\label{eqn_stn_dist_B1}
    \pi_{0,0}&=(1-\rho), \ \ \pi_{1,0}=\rho(1-\rho), \ \ \pi_{0,1}=0 \mbox{ (by \ref{a1}),}\nonumber \\
    {\pi}_{q,0}&= 
    \begin{cases}
        \rho^q(1-\rho)-\frac{m_{q-2}\rho^{A+2}\psi\mu}{m_{A}}, &\mbox{for } 2 \le q < A+2,\\
         \rho^{q}\psi\nu, &\mbox{for }  q \ge A+2, \mbox{ and}
    \end{cases}\nonumber \\
    {\pi}_{q,1}&= 
    \begin{cases}
     \frac{m_{q-1}\rho^{A+2}\psi\mu}{m_{A}}, &\mbox{for } 1 \le q < A+1,\\
         \rho^{q+1}\psi\mu, &\mbox{for } q \ge A+1. 
    \end{cases}
\end{align}

\subsection{Small $\nu$-approximation}\label{subsec_approx}
\hide{
\begin{figure*}
    \begin{minipage}{0.6\textwidth}
        \flushleft
    \includegraphics[trim = {1cm 6cm 3cm 7cm},scale=.35]{piA=3B=1.pdf}
    \caption{Stationary dist. for A=3, B=1}
    \label{fig:transition}
    \end{minipage}%
    \hspace{-19mm}
    \begin{minipage}{0.4\textwidth}
        \flushright   \includegraphics[ trim = {1cm 6cm 3cm 7cm},scale = .35]{picturestatioA=37,B=3.pdf}
\caption{Stationary dist. for A=37, B=3}
        \label{fig:transition2}
    \end{minipage}
\end{figure*}}

\hide{{\color{red}Justification ---  It is clear $\nu$ has to be much smaller than $\mu-\lambda$, basically the queue length should reduce significantly in expected sense by the time of return from Lounge.  We precise use this description to derive approximate analysis for the JQL system; we provide the performance of the system when $\nu \to 0$  with $\nicefrac{\beta}{\nu}$ kept fixed. 
Interestingly   the system   obtained  at this  limit   exactly  matches the JDS system (by force and offering better service, no need to wait in queue later). 
}}

In queueing theory, it is typical to consider fluid (where $\lambda, \mu \to \infty$ keeping $\rho$ fixed) or mean-field (considered when the number of queues or servers are exorbitantly large) approximations to understand the system. In our study, none of the said two regimes are valid, rather a new regime needs to be studied which arises due to the interactions among the customers and the system. In particular, our interest lies in understanding the system when the customers benefit from joining the lounge; this happens when customers find a smaller queue upon return than at arrival. Since customers expect the queue to decrease at approximately $(\mu-\lambda)$ rate, significant reduction within time $\nicefrac{1}{\nu}$ occurs only  if $\nicefrac{(\mu-\lambda)}{\nu} \gg 1$, i.e., if $\nu \ll \mu - \lambda$. Recall that such arguments lead to the optimal customer response in Theorem \ref{lemma_threld}. We focus on this scenario from the system's perspective, using the following approximation:
\begin{align}\label{eqn_cond_nu0}
    \nu \to 0 \mbox{ and } \nicefrac{\beta}{\nu} = \eta,
\end{align}where $\eta > 0$ is a constant; note $B \to \infty$ by above.

When $\nu \to 0$, customers  wait for extremely long times in the lounge with high probability, which would normally deter lounge use. To counter this, we maintain a fixed ratio $\nicefrac{\beta}{\nu}$, allowing $\beta \to 0$ (as a function of $\nu$) at the same rate as $\nu$. This ensures that customers feel extremely comfortable while waiting in the lounge, making them willing to join despite potentially longer waiting times.

Our goal in this sub-section is to study the queueing system with LF near the regime \eqref{eqn_cond_nu0}. Towards this, note that if one varies $\nu$ (and set $\beta = \nu \eta$), keeping every other parameter constant, then each value of $\nu$ leads to a different system with $A = \lfloor\nicefrac{\eta\mu}{\alpha}\rfloor$. Our conjecture is that the performance of the original system parameterized by $\nu$ can be approximated by the system where $\nu = 0$ (henceforth referred to as the `\textit{approximating system}') and $A = \lfloor\nicefrac{\eta\mu}{\alpha}\rfloor$:

\begin{figure*}[htbp]
\vspace*{15pt}
    \begin{minipage}{0.33\textwidth}
        \centering
        \includegraphics[trim = {1cm 0cm 3cm 0.5cm}, clip, scale = 0.16]{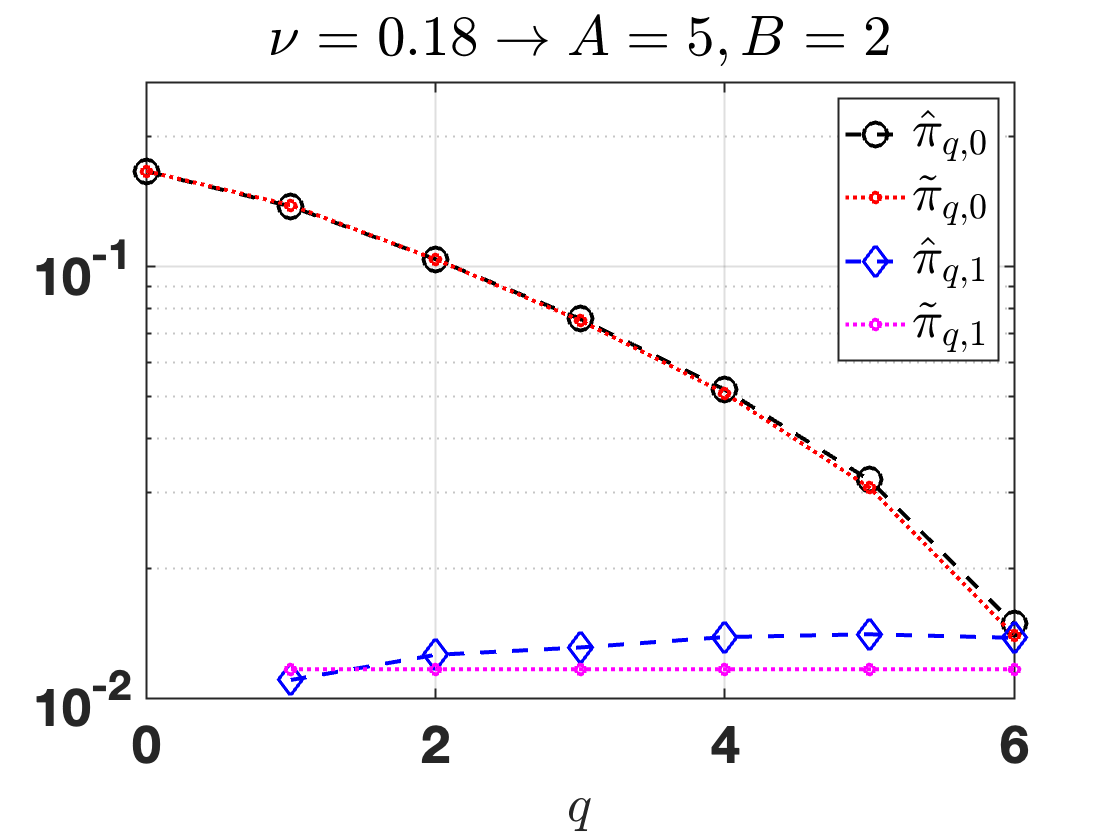}
    \end{minipage}%
    \begin{minipage}{0.33\textwidth}
        \centering
        \includegraphics[trim = {1cm 0cm 3cm 0.5cm}, clip, scale = 0.16]{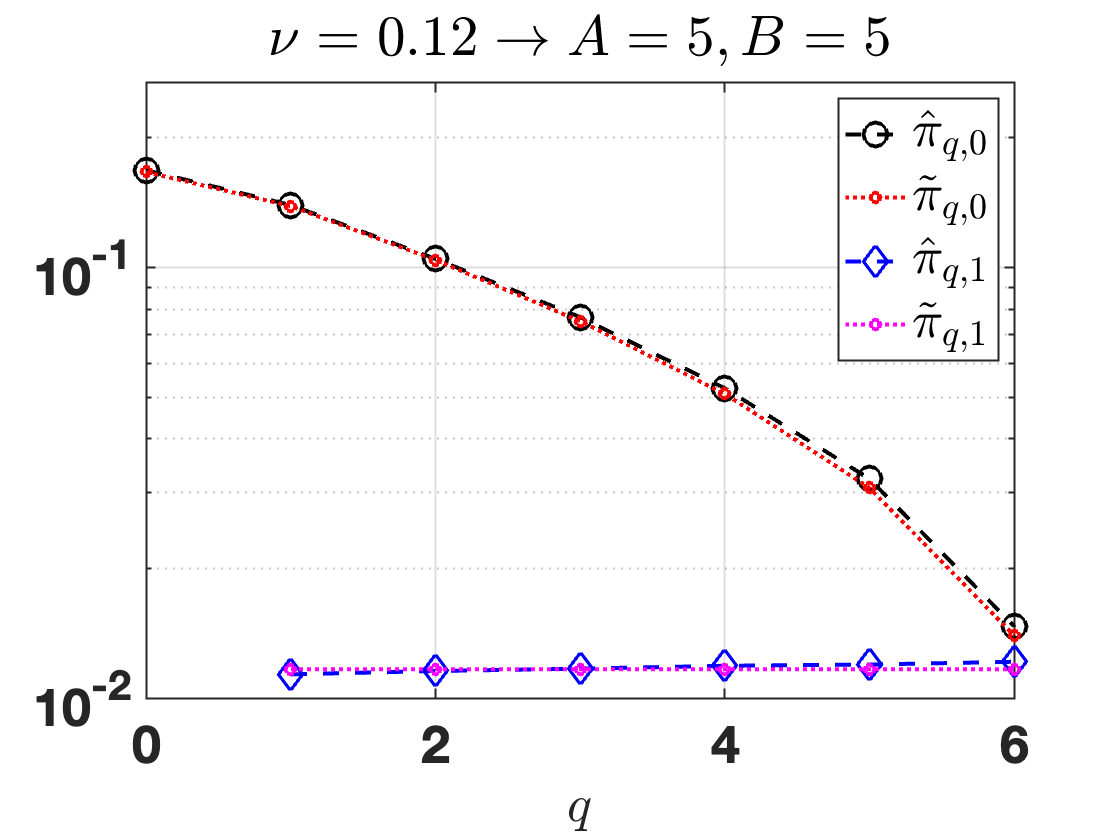}
    \end{minipage}%
    \begin{minipage}{0.33\textwidth}
        \centering
        \includegraphics[trim = {1cm 0cm 3cm 0.5cm}, clip, scale = 0.16]{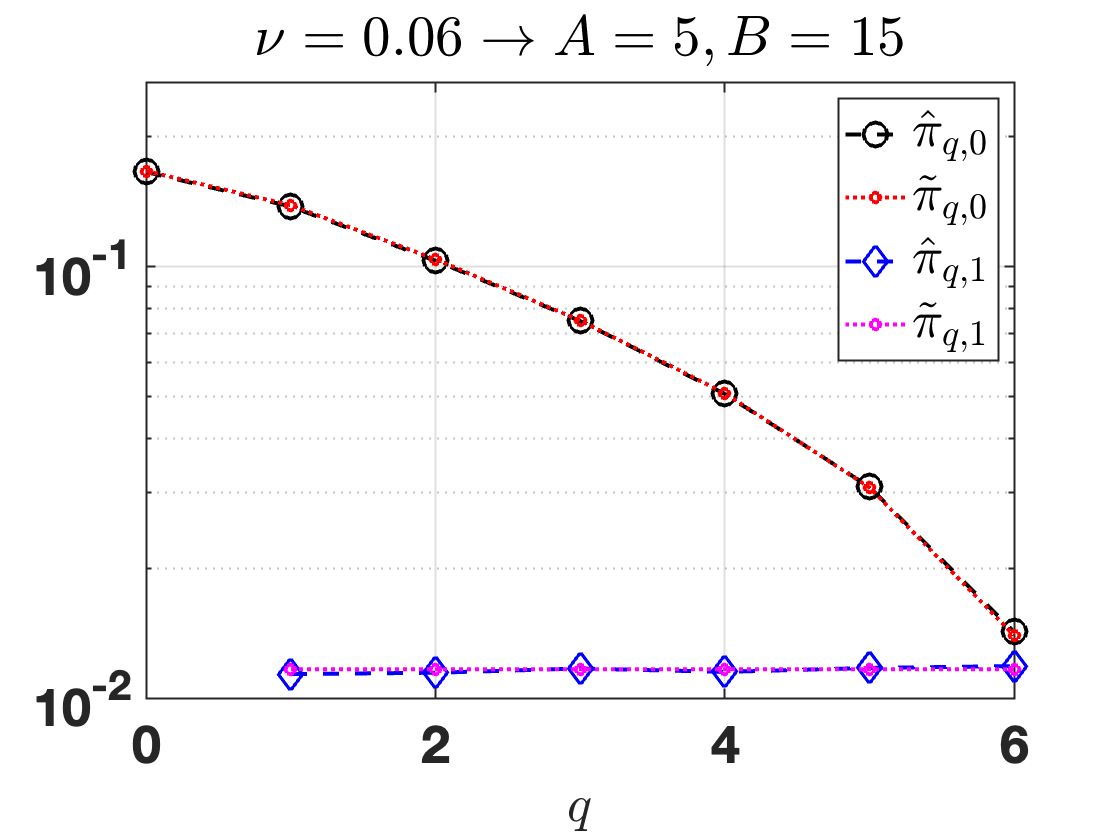}
    \end{minipage}
    \caption{Approximation of $\bpi^{(\nu)}$ improves as $\nu \to 0$; depicted in `log scale' for $l = 0,1$  }
    \label{fig:approx}
\end{figure*}
\begin{conjecture}\label{conj}
    Let $\bpi^{(\nu)} := (\pi_{q, l}^{(\nu)})_{q,l}$ represent the stationary distribution for the system with LF when the transition rate from the lounge to the queue is $\nu$. Then, the following holds:
    \begin{align*}
       \hspace{0.5cm} \sup ||\bpi^{(\nu)} - \bpi^{(0)}|| \to 0, \mbox{ as } \nu \to 0 \mbox{ and } \nicefrac{\beta}{\nu} = \eta. \hspace{0.5cm} \mbox{ \eop}
    \end{align*}
\end{conjecture}
The proof for the above follows by the Maximum Theorem on compact Banach spaces. We skip the proof in this paper (and plan to include it in the journal version of this work). 

In \autoref{fig:approx}, we validate the above conjecture by considering $
\lambda = 6, \mu = 7.2, \alpha = 0.45 \mbox{ and } \eta = 0.35$. Here, we plot the stationary probabilities for the approximating system (denoted by $\widetilde{\pi}_{q,l}$, exactly computed below in \eqref{eqn_stn_prob_approx}) and the Monte Carlo estimates of the same for the original system with $\nu > 0$ (denoted by $\widehat{\pi}_{q,l}$) for $l = 0,1$. One can easily verify that the approximation improves as $\nu$ decreases (and thus, as $B$ increases). However, it is equally crucial to note that the approximation is sufficiently well even when $B$ is small; we use this observation to derive further analysis in sub-section \ref{subsec_num_expt}. We shall see below that $\widetilde{\pi}_{q,l} = 0$ for $q > A+1$ and, thus, such probabilities are not representable on the log scale; for this reason, only the states $(q,l)$ with $q \leq A+1$ are depicted in the figure.

In light of the Conjecture \ref{conj} and above arguments, we restrict our attention to the \textit{approximating system} and analyze the same in detail. Here, the decision rule is modified as follows under Theorem \ref{lemma_threld} (note $B = \infty$):
\begin{align}\label{eqn_threshold_nu0}
a^* = 2 \cdot 1_{\{Q > A\}}, \mbox{ with } A = \lfloor\nicefrac{\eta\mu}{\alpha}\rfloor.
\end{align}

In view of the above decision-rule,
interestingly, the dependency on the lounge congestion disappears. So happens because even if a customer has to wait for a long time in the lounge, it perceives no discomfort; thus, it only has to judge based on the queue length at its arrival instant. In particular, if the queue is long ($Q > A$), it is optimal for a customer not to join the queue. Thus, all states $(Q, L)$ with $Q > A+1$ are transient. For this reason and the fact that $B \to \infty$, the state space for the approximating system is given by:
$$
\mathbb{S}^{(0)} = \{(Q, L): Q\in \{0, 1, \cdots A+1\}, L \in \mathbb{Z}^+\cup\{0\}\},
$$
and the MC simplifies as shown in \autoref{fig:transition2}. It is immediate to see that the balance equations for the approximating system is given by:
\begin{align}
\lambda \pi_{0,0}&=\mu \pi_{1,0}, \label{eqn_approx1}\\
(\lambda+\mu)\pi_{1,l}&=\mu(\pi_{2,l}+\pi_{1,l+1})+ \lambda \pi_{0,0}\indc{l = 0}, \label{eqn_approx2}\\
(\lambda+\mu)\pi_{q,l}&=\lambda\pi_{q-1,l}+\mu \pi_{q+1,l}, \mbox{ for } 2 \le q\le A, \label{eqn_approx3}\\   (\lambda+\mu)\pi_{A+1,l}&=\lambda \pi_{A,l}+\lambda\pi_{A+1,l-1}\ \indc{l > 0}.\label{eqn_approx4}
\end{align}

%
%

In order to find the stationary distribution $(\widetilde{\pi}_{q,l}) := (\pi_{q,l}^{(0)})$ of the approximating system, we handle one layer at a time with respect to $l$ (see \autoref{fig:transition2}). Let us begin with the bottom layer with $l = 0$; here, we derive the $(\widetilde{\pi}_{q,0})$ in three steps:
\begin{enumerate}
    \item[(i)] we note the stationary probabilities for the extreme left states to be $\widetilde{\pi}_{0,0} = 1-\rho$ and $\widetilde{\pi}_{1,0}=\rho(1-\rho)$ (by \eqref{eqn_prob_n_cust}).
    \item[(ii)]  using \eqref{eqn_approx3} and \eqref{eqn_approx4}, we express all $\widetilde{\pi}_{q,0}$, for $1 \leq q \leq A$, in terms of the stationary probability for the extreme right state $(A+1, 0)$, i.e., $\widetilde{\pi}_{A+1,0}$ as follows:
    \begin{align}
        \widetilde{\pi}_{q,0}&=\sum_{i=0}^{A-q+1}\frac{1}{\rho^i} \widetilde{\pi}_{A+1,0}, \mbox{ for all } 1 \le q \le A. \label{eqn_approx_l0}
    \end{align}
    \item[(iii)] we solve for the value of $\widetilde{\pi}_{A+1,0}$ using \eqref{eqn_approx_l0} with $q = 1$:
\begin{align}
    \widetilde{\pi}_{A+1, 0} = \frac{\rho^{A+1}(1-\rho)^2}{1-\rho^{A+1}}, \label{eqn_pi_Aplus1_l0}
\end{align}and then substitute this value in \eqref{eqn_approx_l0} to derive $\widetilde{\pi}_{q,0}$ for all $2 < q \leq A$.
\end{enumerate}
Now, let us consider the second layer with $l = 1$. Here again, we follow the same procedure as above (with slight modifications to cater for the transitions coming from the bottom layer, see \autoref{fig:transition2}). Since $\widetilde{\pi}_{1,1}+\widetilde{\pi}_{2,0} = \rho^{2}(1-\rho)$ (from \eqref{eqn_prob_n_cust}), therefore, the stationary probability for the extreme left state when $l=1$ is  (see \eqref{eqn_approx_l0} and \eqref{eqn_pi_Aplus1_l0}):
\begin{align}
    \widetilde{\pi}_{1,1}=\frac{\rho^{A+2}(1-\rho)^2}{1-\rho^{A+1}}.\label{eq_pi_11_l1}
\end{align}For the rest (i.e., $1 \le q \le A$), we express each $\widetilde{\pi}_{q,1}$ in terms of $\widetilde{\pi}_{A+1,1}$ and $\widetilde{\pi}_{A+1,0}$ (corresponds to the extreme right states of layers with $l=1, 0$ respectively):
\begin{align}
    \widetilde{\pi}_{q,1}=\sum_{i=0}^{A-q+1}\frac{1}{\rho^i}\widetilde{\pi}_{A+1,1}-\sum_{i=0}^{A-q}\frac{1}{\rho^i}\widetilde{\pi}_{A+1,0}. \label{eq_l=2pi}
\end{align}
Clearly, $\widetilde{\pi}_{A+1,1}$ is obtained using above equation with $q = 1$, \eqref{eq_pi_11_l1} and \eqref{eqn_pi_Aplus1_l0}:
\begin{align}
    \widetilde{\pi}_{A+1,1}=\frac{\rho^{A+2}(1-\rho)^2}{1-\rho^{A+1}}.
\end{align}The probabilities $\widetilde{\pi}_{q,1}$ for $1<q<A+1$ can now be easily obtained. 

The procedure depicted above for $l=0,1$ can essentially be followed for higher values of $l$ by using  the following general relationship for any $1 \le q <A+1$ and $l \ge 1$:
\begin{align}
    \widetilde{\pi}_{q,l}&=\sum_{i=0}^{A-q+1}\frac{1}{\rho^i}\widetilde{\pi}_{A+1,l}-\sum_{i=0}^{A-q}\frac{1}{\rho^i}\widetilde{\pi}_{A+1,l-1}.
\end{align}

Finally, the stationary distribution for the approximating system can be summarized as follows:
\begin{align}\label{eqn_stn_prob_approx}
    \begin{aligned}
        \widetilde{\pi}_{0,0}&=1-\rho, \ \
        \widetilde{\pi}_{A+1,l}=\frac{\rho^{A+l+1}(1-\rho)^2}{1-\rho^{A+1}} \mbox{ for } l \geq 0,\\
        \widetilde{\pi}_{q,l}&=
            \begin{cases}
                \frac{\rho^{A+l+1}(1-\rho)^2}{1-\rho^{A+1}}, &\mbox{ for } l > 0, \\
                \frac{(1-\rho)\rho^q}{1-\rho^{A+1}}\left(1-\rho^{A+2-q}\right), &\mbox{ for } l = 0,
            \end{cases}
    \end{aligned}
\end{align}
for all $1 \leq q < A+1$.

\vspace{2mm}
\noindent \textit{Practical relevance of approximating system:} 
Recall that in the approximating system, customers wait indefinitely long in the lounge, receiving service only when\footnote{Since our system possesses the characteristics of M/M/1 queue, thus, a customer which enters the system has to leave.} the queue becomes empty. This scenario mirrors realistic situations like waiting for visa interviews or doctor consultations in a lounge. To our knowledge, such a system remains unstudied.

Specifically, our approximating system resembles one where customers have two options: (1) enter a queue and wait on a FCFS basis, or (2) enter a lounge where the longest-waiting customer is called directly to the server by the controller when the server becomes idle. Customers choosing the second option wait exclusively in the lounge without entering the conventional queue.

We have partially\footnote{In the analysis, we have not implemented the FCFS policy in the lounge.} studied this new system assuming customers accept the second option with no hesitation. However, as we did for the original system, it is essential to investigate customers' behavioral responses to this delay-sensitive routing policy in future research.

\section{Lounge design problem}\label{sec_loungedesign}
We now examine the problem from the perspective of the system who needs to make two key decisions: (i) is it beneficial to design an LF?, and (ii) what is the optimal value of comfort that should be provided to the customers if the LF is made?
One might anticipate that the LF can help alleviate congestion  problem, more so at high load conditions. On the contrary, interestingly, we will show in this section that it is not always beneficial to incorporate an LF even if one ignores the monetary expenses to be incurred for designing the lounge.

Towards this, we consider a two-level optimization problem or a Stackelberg  (SB) problem by designating the system as the leader and the customers as the followers. The leader chooses either to provide an LF with an appropriate comfort factor $\eta := \nicefrac{\beta}{\nu} $, or not to provide it at all. We depart from the standard SB framework by capturing the response of the  customers  (at the lower level) directly via the customer-response  parameters $A=A(\eta) = \lfloor\nicefrac{\eta \mu}{\alpha}\rfloor$ and $B=\lceil\nicefrac{(\mu-\lambda)}{\nu}-A \rceil $ derived in Theorem \ref{lemma_threld}.

The stationary distribution $\bpi = \{\pi_{q,l}\}_{q,l} =:\bpi(A(\eta)) $ of the system-level Markov chain $(Q_t, L_t)$   depends upon the customer response parameter $A(\eta)$ and  determines the long-run average of the congestion cost to be incurred by the system.
 Let $G(\eta)$ represent this  congestion cost defined  in terms of the weighted ($\omega > 0$) stationary second moments of the queue and lounge occupancy  levels: 
\hide{In other words one can view s the response of the customers for the choice $\eta$ of the system and then  thus $G(\eta) = G(\bpi (A(\eta) ) $ represents the congestion cost:

The design of the lounge is characterized by selecting the  comfort parameter  $\eta$, which influences both the system's monetary investment and the resultant congestion cost and budget management. This problem can be formulated as a bi-level optimization problem, or more precisely, as a Stackelberg game. However, in this setting, exceptionally we rely on the customer response function provided by Theorem \ref{lemma_threld} at the lower level.

Let  $M(\eta)$  denote the monetary cost incurred by the system to design a lounge with comfort level $\eta$. Similarly, let $G(\eta)$  denote the total congestion cost experienced in the system including the lounge for a given $\eta$. Customer behavior in response to $\eta$  is captured by:
\begin{align}
    A(\eta) = \eta \frac{\mu}{\alpha}, \quad B = \frac{\mu - \lambda}{\nu} - A(\eta),
\end{align}
as derived in Theorem \ref{lemma_threld}. These parameters govern the dynamics of the system's Markov process $(Q_t, L_t)$, whose resultant stationary distribution  determines the congestion cost. 

In other words one can view $\bpi = \{\pi_{ql}\}_{q,l} =:\bpi(A(\eta)) $ as the response of the customers for the choice $\eta$ of the system and then }
\begin{eqnarray}\label{eqn_cost}
    G(\eta) = \sum_{q,l}  (q^2 + \omega l^2) \pi_{q,l} (A(\eta) ).\label{eq_loungeproblem}
\end{eqnarray}
With no-LF option, 
$
G_o := \sum_q q^2 \rho^q (1-\rho)
$ equals 
the congestion cost as the system is the standard M/M/1 queue without LF. 
Now, the solution to the lounge design problem can be defined either as the solution $\eta^*$  of  the following optimization problem  (if the optimal value $G^* \le  G_o$) or as the no-LF option (if $G_o$ is smaller):
\begin{eqnarray}
G^* := \min \left \{  G(\eta)  \ \big | \  \eta > 0 
    \mbox{ and }  \nicefrac{(\mu-\lambda)}{\nu} - A(\eta) > 0 \right \} . \label{Eqn_G_star}
\end{eqnarray} 
The second constraint in the above ensures that $B > 0$; thus, by Theorem \ref{lemma_threld}, $\eta$ chosen as per \eqref{Eqn_G_star} ensures that the LF parameterized by $\eta$ is acceptable to the customers, i.e., 
$$
P_{\bpi (\eta)} (L > 0) = \sum_{(q,l), l > 0 } \pi_{q,l} (A(\eta))  > 0 .
$$
Observe that the congestion cost \eqref{eqn_cost} depends on $\eta$ only via $A(\eta)$; in other words, it is sufficient to repose the  constrained problem \eqref{Eqn_G_star} with $A$  as the design parameter:
\begin{eqnarray}
G^* =  \min  \left \{ G(A) : A \in \left \{0, 1, \cdots, \lceil\nicefrac{(\mu-\lambda)}{\nu}-1\rceil \right \}    \right \};\nonumber
\end{eqnarray}we denote the optimizer of the above problem as $A^*$. Basically, the choice of comfort factor $\eta$ translates now to that of the occupancy limit in lounge, $B$ (which is negatively related to $\eta$). Thus, we finally solve the following problem:
\begin{eqnarray}
\min\left \{ G^*, \ G_o  \right \}. \label{eq_loungedesign}
\end{eqnarray}
 

\subsection{At high load factors  } \label{subsec_high_load}
We consider a regime of parameters that satisfy $\mu -\lambda \le \nu$. This regime can result if   the system is operating under high load conditions, i.e., $\rho \approx 1$ (note $\rho \geq 1-\nicefrac{\nu}{\mu}$).
From Theorem~\ref{lemma_threld}, for such scenarios, the LF will be acceptable to a positive fraction of customers only when 
$B > 0$, or equivalently when 
$\eta < \nicefrac{(\mu-\lambda)\alpha}{\nu \mu}$. 
Thus, for the specified regime, the lounge design problem \eqref{eq_loungedesign} simplifies to a problem with a binary choice for the system --- choose no-LF (i.e., $B=0$) or choose the LF with $B=1$ (recall $B \leq \lceil\nicefrac{(\mu-\lambda)}{\nu}\rceil \leq 1$) and thus, $A = 0$.  Using \eqref{eqn_stn_dist_B1} with $A=0$, we obtain (recall $G_o  = \sum_{q} q^2 \rho^q(1-\rho)$):
%
\begin{align*}
G(A=0) = \frac{(\nu + \rho \mu ) }{\nu+\mu} G_o+ \frac{ \rho^2\mu}{\nu+\mu}\left(\omega + \frac{1-\rho}{\rho}  \right)  .
\end{align*}
Then, the difference  between $ G(A=0)$ and $G_o$ is:
\begin{align*}
 G(A=0)-G_o    = \frac{  \rho^2\mu  }{\mu+\nu}  \left (\omega  -  \overline{\omega} \right ), \mbox{ where } \overline{\omega} := \frac{3-\rho}{1-\rho}. 
 \end{align*}
Thus, we have  proved the following (recall $\rho <1$):
\begin{thm}{\bf [Optimal LF]} \label{thrm_optimalLF}
Assume ${\mu-\lambda}\leq\nu$. It is optimal for the system to provide the LF with $A^*  = 0$, when $\omega < \overline{\omega}$. Otherwise, it is better not to design the lounge.    \eop
\end{thm}
The above theorem indicates that the system increasingly favors designing the LF under high load conditions, as $\overline{\omega}$ increases with $\rho$. Further, the LF needs to be loaded with facilities such that customers experience no discomfort while waiting there (as $A^* = 0$ implies $\beta = 0$). Interestingly, the customers use the LF only sparingly, i.e., only when lounge is empty ($L < B = 1$) and queue is non-empty ($Q > A = 0$). Overall, \textit{when load factor is high, the system should design an LF with capacity for only one customer such that it feels ultra comfortable there!}

\subsection{Numerical analysis}   \label{subsec_num_expt}
Here we aim to further investigate the lounge design problem using numerical computations. Basically, the goal is to understand if the customers and the system prefer the LF and at which capacity. We set  $\mu = 2.5$ and $\nu = 0.1$. 

To begin with, in \autoref{fig:transition4} with $\omega = 1.2$, the congestion cost $G(A)$ (see \eqref{eq_loungeproblem}) is plotted against the customer response  parameter $A$. The computations for $G(A)$ are done using the stationary distribution derived for the approximating system, see \eqref{eqn_stn_prob_approx}. We consider $A \in 
 \left\{ 0,1, \ldots,  \lceil\nicefrac{(\mu - \lambda)}{\nu} -1\rceil \right\}
$ to ensure that the lounge, if designed by the system, is acceptable to the customers (see \eqref{Eqn_G_star}). 
The figure clearly indicates that the function $G(A)$ is convex for all values of $\rho$ and thus, has a unique minimizer $A^*$ at which it achieves the value $G^*$. Recall $G_o = \sum_q q^2 \rho^q (1-\rho)$ and equals $1.56, 4.21, 13.22$ for $\rho = 0.4, 0.55$ and $0.7$ respectively. It is now easy to see that $G^* < G_o$ for each $\rho$. Conclusively, for the underlying parameters, the system would prefer to design the LF (and customers would also prefer to use it).


\begin{figure}[htbp]
\flushright
     \begin{minipage}[b]{0.5\columnwidth}
     \flushleft
\includegraphics[trim = {0.5cm 1cm 2cm 1cm}, clip,scale=.11]{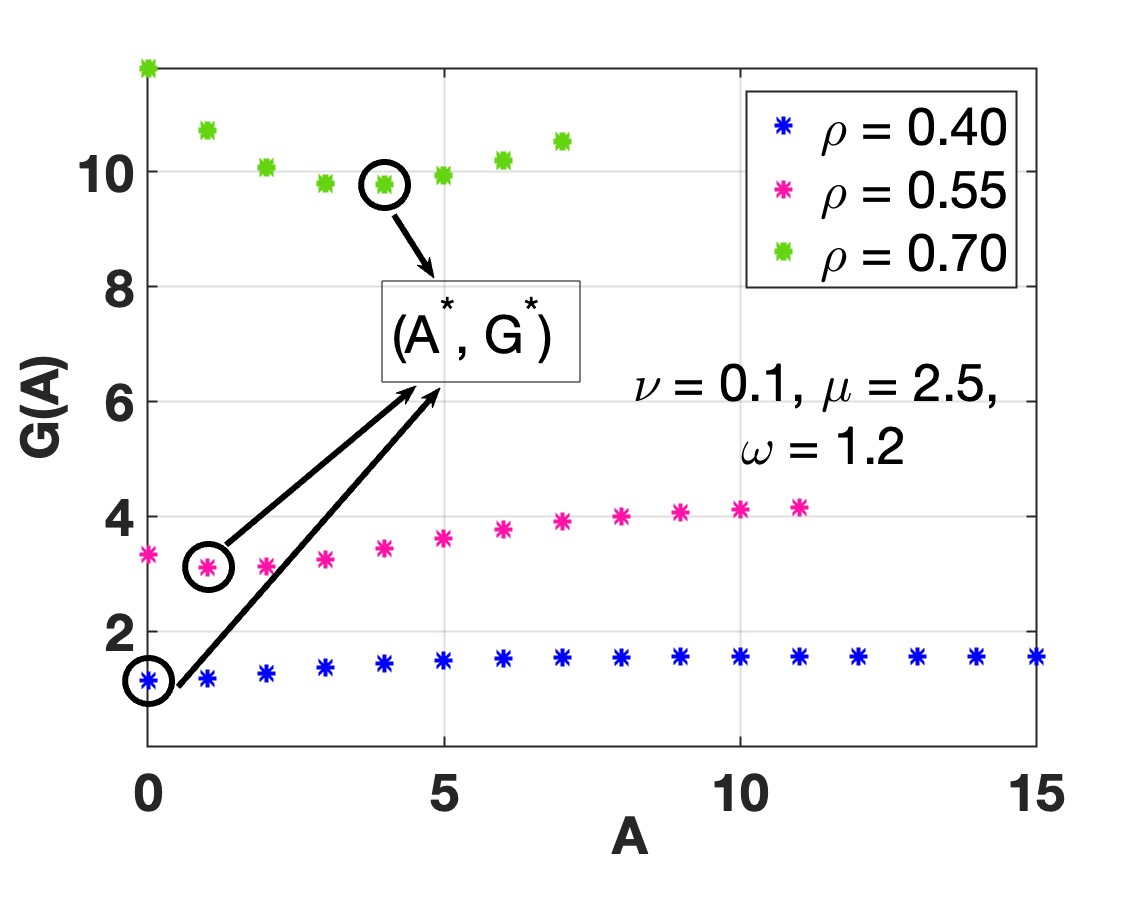}
    \caption{Congestion cost v/s $A$}
    \label{fig:transition4}
    \end{minipage}%
    \hspace{-3mm}
    \begin{minipage}[b]{0.5\columnwidth}
    \flushleft
    \includegraphics[trim = {0cm 0cm 2.4cm 1cm}, clip,scale=.117]{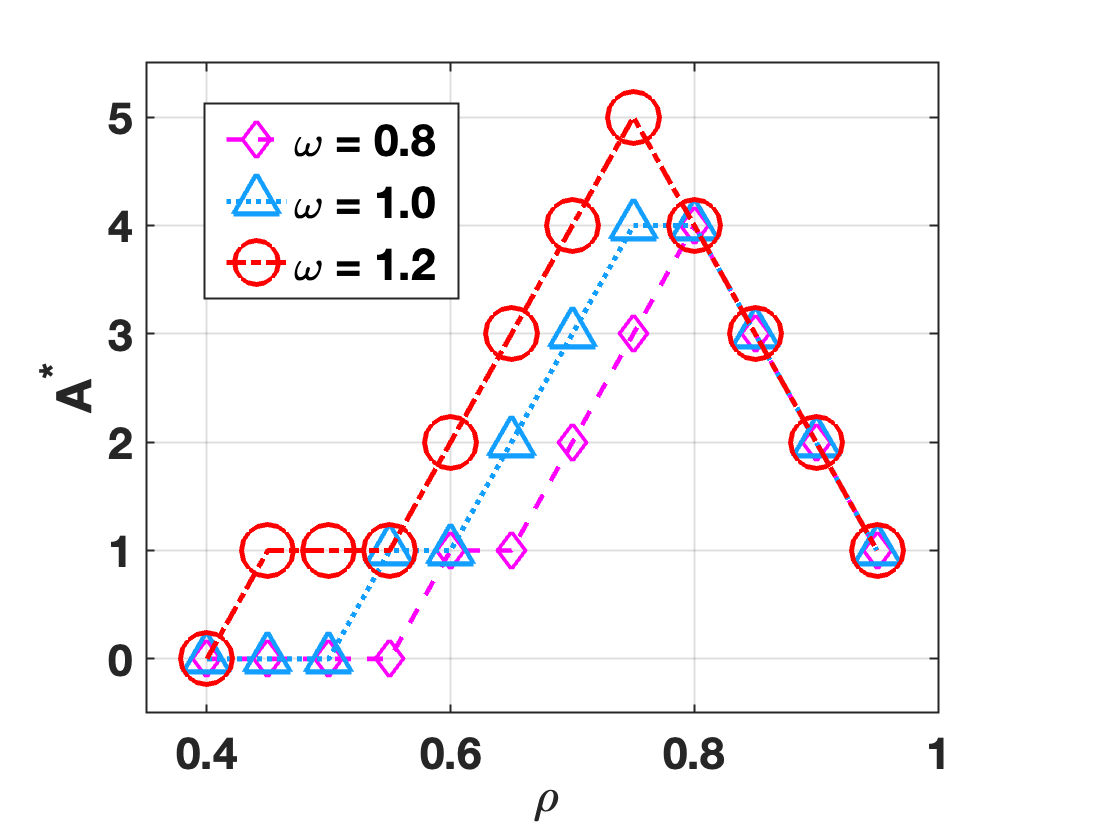}
    \caption{  $A^*$ v/s $\rho$}
         \label{fig:transition5}
         \end{minipage}
\end{figure}

In \autoref{fig:transition5}, 
we plot $A^*$ as a function of $\rho$ for different values of $\omega$. Fix a value of $\omega$ and observe that $A^*$ increases initially and then declines as a function of $\rho$. Importantly, for lower load factors, $A^*$ is small and this results in: (i) the system preferring to design a large LF, and (ii) the customers also being inclined to wait in the LF (as $B(A^*) =\lceil \nicefrac{(\mu-\lambda)}{\nu} - A^*\rceil$ is large). Surprisingly, these trends are reversed when $\rho$ is high (this observation perfectly aligns with the conclusion drawn from Theorem \ref{thrm_optimalLF}); so is true even when $A^*$ reduces for $\rho > 0.75$ because the first term in $B(A^*)$ reduces sharply while the second term decreases slowly with $\rho$.  

\autoref{fig:transition5} also confirms the intuition that if the system is more concerned about the congestion in the lounge (i.e., higher $\omega$) then it will design an LF with lower occupancy; note $A^*(\omega_1; \rho) \ge A^*(\omega_2; \rho)$ (and thus, $B(A^*)(\omega_1; \rho) \le B(A^*)(\omega_2; \rho)$) for $\omega_1 > \omega_2$ and fixed $\rho$.

\section{Conclusions}
Lounges are commonly seen nowadays in many settings and are probably introduced to reduce congestion in high-traffic conditions. Surprisingly, our work indicates a paradoxical behaviour: neither the system prefer to design, nor the customers prefer to use the lounge facility (LF) in the said regime. This result reemphasize the fact that adding resources may not always be optimal (as is seen in the Braess paradox).

Basically, we considered an M/M/1 queue with an additional LF that customers can use to wait and relax before moving to the queue. The decision to join the queue/lounge is based on pessimistic anticipation of future congestion. We show that the customers use the lounge when the queue is long and the lounge is relatively empty. 
While studying the lounge design problem from the system perspective, under high traffic conditions, we illustrate that the system would  be interested in designing the LF but surprisingly with a low
occupancy level. However, in such a regime, the customers are not too keen to wait in the lounge and use it only when it is empty. Unexpectedly, both customers and the system are in favor of a large lounge under low traffic conditions. These results are even more interesting as they hold when the system focuses only on the congestion and not the money. 

In future, we propose to study the system with more realistic irrational behaviors to investigate the impact of non-rationals on the equilibrium strategies of rationals.





\bibliographystyle{IEEEtran}
\bibliography{references}
\newpage
\hide{We begin by considering a queuing model where customers have the option to either directly join the service queue or enter a lounge area. Customers who choose the lounge wait for an exponentially distributed amount of time before rejoining the system to receive service according to a First-Come-First-Served (FCFS) discipline. However, under high-traffic conditions, this model yields a negative result: customers tend to avoid the lounge due to the relatively higher associated cost, thereby rendering the lounge option ineffective. Additionally, the model becomes analytically intractable for general values of the lounge capacity $B$. Therefore, we analyze the stationary distribution for small values of $B$, specifically $B = 1$  and  $B = 2$.

To facilitate tractable analysis, we next consider an approximation in the regime where $ \nu \to 0 $ and  $\beta \to 0$  such that the ratio $\nicefrac{\beta}{\nu}$ converges to a fixed constant. In this limiting case, the effective lounge capacity tends to infinity $( B \to \infty )$. 
As a result, the lounge rarely reaches its capacity limit, and the system behaves as if the lounge has infinite capacity. Thus, in the limiting case, we approximate the system by taking $B \to \infty$ significantly simplifying the analysis.) }


\hide{
{\color{blue} Key points:

One of the main purposes of lounge design is probably  to handle congestion  (in waiting room) at high load conditions. Our analysis illustrates that the system would also be interested in  designing such a facility but with low occupancy level (unless  the constraint on lounge congestion is high). However the customers are not too keen. They sparsely use it and at best  would  use the facility only when it is empty. 

\hide{\begin{enumerate}
    \item \textbf{Customer Response Based on Queue and Load Conditions:} \\
    Customers decide whether to utilize the lounge facility based on two main thresholds:
    \begin{itemize}
        \item \textit{Threshold:} If the main queue length is too high and the lounge occupancy is within an acceptable bound, customers prefer to enter the lounge.
        \item \textit{Load Factor Threshold (Comfort-Driven Behavior):} When the system load factor \( \rho = \frac{\lambda}{\mu} \) exceeds a threshold determined by a comfort parameter \( \eta \), customers choose not to enter the lounge if the lounge is not too comfortable. Otherwise, they enter with a strictly positive probability.
    \end{itemize}

    \item \textbf{High Load Factor Regime (\( \mu - \lambda < \nu \)):} \\
    In cases of high load and limited service slack:
    \begin{itemize}
        \item Providing a lounge is optimal from both the customer response and system performance perspectives, provided the weight assigned to lounge congestion cost exceeds a critical threshold.
        \item Interestingly, under such conditions, the optimal lounge occupancy limit \( B \) reduces to its minimal value of 1, indicating that even minimal lounge capacity can be effective if well-regulated.
    \end{itemize}

    \item \textbf{Low Load Factor Regime (\( \mu - \lambda \gg \nu \)) with High Customer Patience:} \\
    When traffic is light and customers are highly patient:
    \begin{itemize}
        \item The lounge facility remains beneficial if the system places sufficient weight on congestion cost in the lounge and system wants to provide lounge facility with high occupancy level and more customers also willing to use lounge.
        \item As the load factor increases within this regime, the optimal congestion control parameter \( A \) increases, which leads to a decrease in the lounge occupancy limit \( B \).
        \item Thus, under lower load conditions, the lounge can accommodate more customers, encouraging lounge use from both system optimization and customer incentive perspectives. But for higher load factor, neither system prefers to provide high lounge occupancy nor customers willing to use lounge facility. 
    \end{itemize}
\end{enumerate}}
Finally, we want to extend our analysis to a more realistic setting where a fraction of customers are non-rational. These customers do not optimize their decisions independently but instead imitate the choices made by others. Investigating the dynamics of such systems—and particularly how the behavior of non-rational customers influences the equilibrium strategies of rational agents—presents an intriguing direction for future study.}}


\hide{

}

 %
\newpage

\end{document}